\documentclass{article}

\title{Local H\"{o}lder continuity property of the Densities of Solutions of SDEs with Singular Coefficients
\footnote{This research has been supported by grants of the Japanese government and it profited from fruitful discussions with Stefano de Marco.}
}
\author{Masafumi Hayashi\\
 University of the Ryukyus, and Japan Science and Technology Agency\\
 Department of Mathematical Sciences, Faculty of Science,
 \\ Nishihara-cho, Okinawa 903-0213, Japan.\\
E-mail:hayashim6@gmail.com\\
\\
 Arturo Kohatsu-Higa\\
Ritsumeikan University and Japan Science and Technology Agency\\
Department of Mathematical Sciences\\
1-1-1 Nojihigashi,\\
Kusatsu, Shiga, 525-8577, Japan \\
E-mail: arturokohatsu@gmail.com\\
\\
G\^o Y\^uki\\
Ritsumeikan University and Japan Science and Technology Agency\\
Department of Mathematical Sciences\\
1-1-1 Nojihigashi,\\
Kusatsu, Shiga, 525-8577, Japan\\
E-mail: go.yuki153@gmail.com
}

\usepackage{amsfonts}
\usepackage{amsmath}
\usepackage{amssymb}
\usepackage{amscd}
\usepackage{amstext}
\usepackage{amsthm}
\usepackage{mathrsfs}
\usepackage{enumerate}
\usepackage{color}
\usepackage[mathscr]{eucal}

\theoremstyle{plain}
\newtheorem{theorem}{Theorem}
\theoremstyle{plain}
\newtheorem{prp}{Proposition}
\theoremstyle{plain}
\newtheorem{lem}{Lemma}
\theoremstyle{plain}
\newtheorem{cor}{Corollary}
\theoremstyle{plain}
\newtheorem{remark}{Remark}
\theoremstyle{plain}
\newtheorem{definition}{Definition}

\begin{document}
\maketitle
\begin{abstract}
We prove that the weak solution of a uniformly elliptic stochastic differential equation with locally smooth  diffusion coefficient and H\"{o}lder continuous drift has a H\"{o}lder continuous density function. 
This result complements recent results of Fournier-Printems \cite{F1}, where the density is shown to exist if both coefficients are H\"{o}lder continuous and exemplifies the role of the drift coefficient in the regularity of the density of a diffusion.
\end{abstract}

Key words: Malliavin Calculus, non-smooth drift, density function.

2010 Mathematics Subject Classification: Primary 60H07, Secondary 60H10 .

\section{Introduction}

Malliavin calculus is well known as a method to prove the regularity of a solution of a SDE (stochastic differential equation).
Especially,  if we assume that the coefficients of a hypoelliptic SDE are bounded functions with bounded 
derivatives of any order, then the solution has a smooth density (see, for example, Nualart\cite{N1}). In recent years, one of the directions in this area is to develop tools to deal with the case of non-smooth coefficients. 

In this article, we consider the one dimensional SDE of the form $dX_t=\sigma(X_t)dB_t+b(X_t)dt$ on a probability space $(\Omega,\mathcal F,Q),$ where $\{B_t\}_{0\le t}$ is a one dimensional standard Brownian motion.
The main purpose of this paper is to prove the local smoothness of the density of the SDE under some weak assumptions on the drift  coefficient $b$. 

Our assumptions, roughly speaking, are local boundedness of the coefficients, H\"older continuity of $b$, uniformly ellipticity and local smoothness of $\sigma$. 
More details about the assumptions will be given later. 

Under these assumptions, 
we will see that the density of the solution of the above SDE exists on the set in which $\sigma$ is smooth. 
Furthermore, we also show that the density is $\gamma$-H\"older continuous,
 with $\gamma \in (0,\alpha )$ and $\alpha$ is the exponent of the H\"older continuity of $b$.
This shows that the drift coefficient may be a determining factor in the regularity of the density.

Some related results have already been obtained for this problem, for example, 
Fournier and Printems \cite{F1} proved in the case that $\sigma$ is $\alpha$-H\"older continuous with $\alpha>\frac{1}{2}$ and 
$b$ is at most linear growth
then the density of $X_t$ exists. 
Their approach is very simple. 
The key idea is to consider the following random variable which approximates $X_t$; 
$
Z_{\varepsilon}:=X_{t-\varepsilon} + \sigma(X_{t-\varepsilon})(B_t-B_{t-\varepsilon})
$ 
for $\varepsilon\in(0,1)$ and 
using some classical lemmas about the existence of the density and conditions of the coefficients. 
In that case, they showed the existence of the density on the set $\{x\in\mathbb R;\sigma(x)\neq 0\}$. 
A careful analysis of their method shows that the argument for the proof can not be used to obtain any further properties of the density (such as the H\"{o}lder continuity of the density).

For a multi-dimensional SDE whose coefficients depends on time, 
Kusuoka \cite{K1} introduced a space denoted by $V_h$ which is larger than the usual Sobolev space and showed the relation 
between the space $V_h$ and absolute continuity of random variables. 
According to \cite{K1}, one can show the existence of the density of $X_t$ on the set 
$\{x\in\mathbb R;\sigma(x)\neq 0\}$ when the coefficients are bounded, $\sigma$ is twice continuously differentiable on 
$\{x\in\mathbb R;\sigma(x)\neq 0\}$ and $b$ is Lipschitz continuous on $\mathbb R$. 

Our result uses a probabilistic approach to the regularity problem
of fundamental solution to  parabolic equations.
In the theory of parabolic equations, there are some regularity results which we briefly compare here.
In \cite{Friedman}, one can find some classical results on the existence and regularity of fundamental solutions of parabolic equations under global H\"{o}lder continuity 
assumptions on the coefficients of the parabolic equation.
In particular, the H\"{o}lder continuity of coefficients yields higher order smoothness 
of the solution to parabolic equations.

In the modern theory of parabolic equations, these equations are solved in
Sobolev spaces and by using embedding theorems, one can find a modification of a solution such that 
this solution might have Holder continuous derivatives (see \cite{Lady} or \cite{Krylov}).
Thie arguments in this approach 
are somewhat global.
 
On the other hand, in this paper, we focus our attention on the local regularity problem:
Does the local regularity of coefficients yield the same property of solution to parabolic equations?
In particular, except for the existence and uniqueness (in law) of weak solution to stochastic
differential equation, our assumptions are restricted only on a neighborhood of some point.
In \cite{Cherny}, the reader can find some sufficient conditions so that the stochastic differential equation under consideration admits a unique weak solution. 

The main tool of our approach is Malliavin calculus, but in general, due to our local hypotheses, 
the stochastic process $X$ will not be differentiable in the Malliavin sense. 
To solve this problem, we use Girsanov's theorem in order to reduce our study to the solution of the equation 
$dX_t=\sigma(X_t)dW_t$ where $W$ is a new Brownian motion under a new probability measure $P$.
In order to deal with the local smoothness of the diffusion coefficient, 
we use stopping times in order to introduce a localization argument.
This localization will allow us to change the process $X$ by a regularized version $\bar{X}$ for which Malliavin Calculus is applicable.

The remaining problem is how to deal with the change of measure which contains the non-smooth function $b$ 
which implies that this random variable is non differentiable. 
For this reason, we introduce an approximation of the change of measure which is differentiable. 
Finally, to end the argument we only need to measure the distance between the change of measure and its approximation by using 
the H\"older property of $b.$

As in \cite{F1}, we believe that the method introduced here can be generalized to other situations such as SDE's with random coefficients or L\'evy driven SDE with Brownian component. 
For examples of applications of the results obtained here, we refer the reader to \cite{F1} and \cite{S1}.

\section{Preliminaries and Notation}
In  this chapter, we introduce some notations and give a brief introduction to Malliavin calculus. 
\subsection{Some Basic Notations}
For $n\in\mathbb N,$ we denote by $C_b^{n}$ the class of bounded and $n$-times continuously differentiable functions 
with bounded derivatives defined in $\mathbb R$ taking values in $\mathbb R$.
Similarly, we define $C_b^{\infty}$ as the class of bounded smooth functions defined in $\mathbb R$ and taking values in $\mathbb R$ with bounded derivatives of any order and $C_p^{\infty}$ as the class of all infintely continuously differentiable functions defined in $\mathbb R$ and taking values in $\mathbb R$ such that the function and its derivatives have at most polynomial growth. 
For a bounded function $f:$ $\mathbb R \rightarrow \mathbb R$, we denote by $\|f\|_{\infty}$ the supremum norm of $f$.

Let $X$ be a random variable on the probability space $(\Omega,\mathcal F,P).$
For $1\le p<\infty,$ we denote
\begin{equation*}
\|X\|_{L^p(P)}:=E_P[|X|^p]^{\frac{1}{p}},
\end{equation*} 
where $E_P[X]$ means the expectation of $X$ with respect to $P$. 
\subsection{Brief Introduction to Malliavin Calculus}
\label{intM}
Now we turn to introduce Malliavin calculus. 
For the proofs of the following results and more details about Malliavin calculus, see \cite{N1}. 
In this chapter, we abbreviate $\|\cdot\|_{L^p(P)}$ by $\|\cdot\|_{L^p}$. 

Fix $T>0$. 
For any measurable function $h\in\mathcal H:=L^2([0,T];\mathbb R),$ we denote  its stochastic integral by
\begin{align*}
W(h):=\int_0^Th(s)dW_s,
\end{align*}
where $\{W_t\}_{t\ge 0}$ is a one dimensional Brownian motion. 

Define 
\begin{align*} 
\mathcal S:=\{F:F=f(W(h_1),\cdots ,W(h_n));h_1,\cdots ,h_n\in\mathcal H,
f\in C_p^{\infty}(\mathbb R)\}.
\end{align*}

For $F\in\mathcal S$ and $t\in[0,T]$, we define the $H$-derivative $($or the Malliavin derivative$)$ as
\begin{align*}
D_tF:=\sum_{i=1}^{n}\partial_if(W(h_1),\cdots ,W(h_n))h_i(t)
\end{align*}
and for $k\in\mathbb Z_+$ and $p\ge 1$, define the norm $\|\cdot\|_{k,p}$ by
\begin{align*}
&\|F\|_{k,p}:=
\left\{
E[|F|^p] + \sum_{j=1}^kE[\|D^jF\|_{\mathcal H^{\otimes j}}^p] 
\right\}^{\frac{1}{p}}, \\
&{\rm where} \\
&\|D^jF\|_{\mathcal H^{\otimes j}}:=\int_0^T\cdots\int_0^T|D_{s_1}\cdots D_{s_j}F|^2ds_1\cdots ds_j.
\end{align*}
As usual, $\|F\|_{0,p}:=\|F\|_{L^p}.$ 

We will denote by $\mathbb D^{k,p}$ the completion of $\mathcal S$ with respect to the norm $\|\cdot\|_{k,p}$ and by $\mathbb D^{\infty}:=\cap_{k,p}\mathbb D^{k,p}$. 
Similarly, for a Hilbert space $V$ and $V$-valued random variables, one can define $\mathbb D^{k,p}(V)$ and $\mathbb D^{\infty}(V):=\cap_{k,p}\mathbb D^{k,p}(V)$. 
In particular, for a $\mathbb R$-valued stochastic process $\{u_s\}_{0\le s \le T}$, we define the norm 
\begin{align*}
\|u\|_{k,p}:=\left\{
E[\|u\|_{\mathcal H}^p] + \sum_{j=1}^kE[\|D^ju\|_{\mathcal H^{\otimes j}}^p] 
\right\}^{\frac{1}{p}}.
\end{align*}
We define the Skorokhod integral, as the dual operator of $D$ and denote it by $\delta.$ 

Let $\{\mathcal F_s\}_{0\le s\le T}$ be the filtration generated by our Brownian motion $\{W_s\}_{0\le s\le T}$. 
It is a well known fact that for $\{\mathcal F_s\}_{0\le s\le T}$-adapted $L^2$ stochastic process $\{u_s\}_{0\le s\le T}$, its Skorokhod integral coincides with its It\^o integral. That is,
\begin{align*}
\delta(u)=\int_0^Tu_sdW_s.
\end{align*}

Moreover, if $\{u_s\}_{0\le s\le T}$ belongs to the domain of $\delta$ $($for example, $u\in\mathbb D^{1,2}(\mathcal H))$, $F\in\mathbb D^{1,2}$ and they satisfy $E[F^2\int_0^Tu_s^2ds]$ is finite, then 
\begin{align*}
\delta(Fu_t)=F\delta(u)-\int_0^t(D_sF)u_sds
\end{align*}
is hold provided the right hand side of the above equation is square integrable. 

For $F=(F^1,\cdots ,F^d)\in(\mathbb D^{1,2})^d$, define the $d\times d$-matrix $M_F$ by
\begin{align*}
M_F^{ij}:=\langle DF^i,DF^j\rangle_{\mathcal H}.
\end{align*}
This $M_F$ is called Malliavin covariance matrix. 
The random vector $F$ is non-degenerate if for any $p\ge 1$, 
\begin{align*}
E[(\det M_F)^{-p}]<+\infty .
\end{align*}
The following proposition, so called integration by parts formula (in Malliavin's sense), plays an important role in this paper. 
\begin{prp}\label{Int_By_Parts}
(Integration by parts formula)

Let $F,G\in\mathbb D^{\infty}$ be nondegenerate and $\varphi\in C_p^{\infty}$. 
Then for any $n\in\mathbb N$, there exists random variable $H_n\in\mathbb D^{\infty}$ such that
\begin{align*}
E\left[\varphi^{(n)}(F)G\right]=E\left[\varphi(F)H_n(F,G)\right].
\end{align*}
Moreover $H_n$ is recursively given by 
\begin{align*}
&H_1(F,G):=\delta(G{M_F}^{-p}DF)\\
&H_k(F,G):=H_1(F,H_{k-1}(F,G))\ {\rm for}\ 2\le k\le n
\end{align*}
and for $1\le p<q<+\infty$, we have
\begin{align*}
\|H_n(F,G)\|_{L^p}\le c_{p,q}\|{M_F}^{-1}DF\|_{n,r2^{n-1}}^n\|G\|_{n,q}
\end{align*}
where $r$ satisfies that $\frac{1}{p}=\frac{1}{q}+\frac{1}{r}$ and $c_{p,q}$ is a constant depends only on $p$ and $q$.
\end{prp}
\section{Preparatory Lemmas}
The basic argument to study the density of a random variable follows from the study of its characteristic function. The first basic result is the following.
\begin{theorem}\label{Levy}$($L\'evy's inversion theorem$)$
Let $(\Omega,\mathcal F,P)$ be a probability space and $X$ be a $\mathbb R$-valued random variable defined on that space.
If  
$\displaystyle
\varphi(\theta):=E[e^{i\theta X}],
$ 
the characteristic function of the $X$,
belongs to $L^1(\mathbb R)$, then 
$f_X$, the density function of the law of $X$, exists and is continuous.
Moreover, 
\begin{equation*}
f_X(x)=\frac{1}{2\pi}\int_{-\infty}^{+\infty}e^{-i\theta x}\varphi(\theta)d\theta
\end{equation*}
for any $x$ in $\mathbb R.$
\end{theorem}
This result is very well-known result which is called ``L\'evy's inversion theorem" for the proof of this, see e.g. \cite{W1}.
The following corollary gives us a more precise criterion for the H\"{o}lder continuity of the density.
\begin{cor}\label{Cor1}
Let $X$ be a random variable under the same setting as in Theorem $\ref{Levy}$ and $\varphi$ be its characteristic function. 
Assume that the following inequality holds for some positive constant $C$ and $0<\gamma <1$.
\begin{equation*}
|\varphi(\theta)|\le 1\wedge(C|\theta|^{-(1+\gamma)}).
\end{equation*}
Then the density function of the law of $X$ exists and is $\alpha$-H\"older continuous for any $0<\alpha<\gamma$.
\end{cor}

\begin{proof}
Let $\alpha\in(0,\gamma)$.
The existence and continuity of the density immediately follows by Theorem $\ref{Levy}$.
We only show that the density is $\alpha$-H\"older continuous.
Let $f_{X}$ be the density of the law of $X$. 
Then by Theorem $\ref{Levy}$, we have
\begin{align*}
|f_X(x)-f_X(y)|&\le \frac{1}{2\pi}\int_{-\infty}^{+\infty}|e^{-i\theta x}-e^{-i\theta y}||\varphi(\theta)|d\theta \\
&\le
\frac{1}{2\pi}\int_{-\infty}^{+\infty}|e^{-i\theta y}||e^{-i\theta(x-y)}-1||\varphi(\theta)|d\theta\\
&\le
\frac{C_{\alpha}}{2\pi}\int_{-\infty}^{+\infty}|\theta x - \theta y|^{\alpha}|\varphi(\theta)|d\theta\\
&=
|x-y|^{\alpha}\frac{C_{\alpha}}{2\pi}\int_{-\infty}^{+\infty}|\theta|^{\alpha}|\varphi(\theta)|d\theta.
\end{align*}
By the hypothesis, the last integral is finite.
Hence, $f_X$ is $\alpha$-H\"older continuous. 
\end{proof}
Now we define the notion of local density function.
\begin{definition}\label{Def1}
Let $\varepsilon$ be a positive number and $y_0\in\mathbb R$. 
The random variable $X$ has a (local) density function $p$ on the set $B_{\varepsilon}(y_0)$ if 
\begin{align*}
E[f(X)]=\int_{\mathbb R}f(x)p(x)dx
\end{align*}
holds for any bounded continuous function $f$ whose support in $B_{\varepsilon}(y_0).$ 
\end{definition}
\begin{remark}\label{Rem_Def}
The above function $p$ corresponds to the density function of $X$ on the set $B_{\varepsilon}(y_0)$ provided $X$ has a density function, but $p$ may exist when $X$ does not have a density function. 
For example, if $X=0$ almost surely, then $X$ clearly does not have a density function. 
However, for any $y_0\in\mathbb R\setminus\{0\}$ and $0<\varepsilon<|y_0|$, the constant function $p=0$ satisfies the above definition.
\end{remark}
Although Corollary $\ref{Cor1}$ gives us a useful criterion about the global existence and continuity of the 
density function, we need another lemma which is used to show the local existence of the density function.
\begin{lem}\label{Lem1}
Assume that $X$ is a random variable under the same setting as in Theorem $\ref{Levy}$. 
Let $\varepsilon>0$ and $\phi_{\varepsilon}$ be an element of $C_{b}^{\infty}$ which satisfies that
\begin{align*}
1_{B_{\varepsilon}(0)}\le \phi_{\varepsilon} \le 1_{B_{2\varepsilon}(0)}.
\end{align*}
Fix $y_0\in\mathbb R$ and set 
$m_0:=E[\phi_{\varepsilon}(X-y_0)]$.  
If $m_0>0$, we define $\mathcal L_{y_0}$ as the probability measure on $\mathbb R$ such that
\begin{equation*}
\int_{\mathbb R}f(y)\mathcal L_{y_0}(dy)
=
\frac{1}{m_0}E[f(X)\phi_{\varepsilon}(X-y_0)],
\end{equation*}
for all continuous and bounded function $f.$ 

If $\mathcal L_{y_0}$ possesses a density $\tilde p_{y_0}$ then 
$p_{y_0}:=m_0\tilde p_{y_0}$ is the density function of $X$ on $B_{\varepsilon}(y_0)$. 

If $m_0=0,$ then the constant function $\tilde p_{y_0}=0$ is a density function of $X$ on $B_{\varepsilon}(y_0)$ even if $\mathcal L_{y_0}$ does not have a density.
\end{lem}
\begin{proof}
Let $m_0>0$ and $f$ be a continuous and bounded function whose support is a subset of $B_{\varepsilon}(y_0)$. 
By the definition of $p_{y_0},$ we have
\begin{align*}
\int_{\mathbb R}f(y)p_{y_0}(y)dy
&=m_0\int_{\mathbb R}f(y)\tilde p_{y_0}(y)dy\\
&=m_0\int_{\mathbb R}f(y)\mathcal L_{y_0}(dy)\\
&=E[f(X)].
\end{align*}
This implies that $p_{y_0}$ is a density function of $X$ on $B_{\varepsilon}(y_0)$. 

On the other hand, if $m_0=0$ then it is clear that 
\begin{align*}
E[f(X)]=0.
\end{align*}
Therefore $p_{y_0}=0$ is a density function of $X$ on $B_{\varepsilon}(y_0)$.
\end{proof}

\begin{remark} 
The function $\phi_\varepsilon$ in Lemma \ref{Lem1} can be constructed as follows. 
Let $a\in (\varepsilon,2\varepsilon)$. 
Define the function 
\begin{align*}
f_{a,2\varepsilon}(x):=
\left\{ 
\begin{array}{ll}
\exp(\frac{1}{x-2\varepsilon}-\frac{1}{x-a});\ \ &\text{ for }x\in(a,2\varepsilon) \\
0;&\text{ for }x\notin (a,2\varepsilon). \\
\end{array}
\right.
\end{align*}
and 
\begin{align*}
g_{a,2\varepsilon}(x):=\frac{\int_a^x f_{a,2\varepsilon}(y)dy}{\int_a^{2\varepsilon} f_{a,2\varepsilon}(y)dy}.
\end{align*}
Then $g_{a,2\varepsilon}\in C_b^{\infty}$ and 
\begin{align*}
g_{a,2\varepsilon}=
\left\{ 
\begin{array}{ll}
0\ \ &(x\le a) \\
1&(x\ge 2\varepsilon). \\
\end{array}
\right.
\end{align*}
Hence $\phi_{\varepsilon}$ may be defined as 
$\phi_{\varepsilon}:=g_{-2\varepsilon,-a}(1-g_{a,2\varepsilon})$.
\end{remark}
\ 

Before stating and proving our main result, we remind the reader that according to Corollary $\ref{Cor1}$ and Lemma $\ref{Lem1}$, if 
\begin{equation}\label{Goal}
|E_Q[e^{i\theta X_t}\phi_{\varepsilon}(X_t-y_0)]|\le 1\wedge(C|\theta|^{-(1+\gamma)})\ (\forall |\theta|\ge 1)
\end{equation}
holds for some positive constants $C$ and $\gamma$, then for any $\gamma'\in(0,\gamma)$ the density function of the $X$ exists and is $\gamma'$-H\"older continuous on $B_{\varepsilon}(y_0)$ at time t.
Here, $\phi_{\varepsilon}$ is an element of $C_b^{\infty}(\mathbb R)$ which satisfies the conditions of Lemma $\ref{Lem1}$. 

\section{Main result}
Let $(\Omega,\mathcal {F}, \{\mathcal {F}_t\}_{t\ge 0},Q)$ be a probability space, 
where $\{\mathcal {F}_t\}_{t\ge 0}$ is the filtration generated by the one dimensional standard Wiener process 
$B:=\{B_t\}_{t\ge 0}$ 
on $(\Omega,\mathcal {F},Q)$.
Consider the following SDE;
\begin{equation}\label{SDE}
X_t=x_0 + \int_0^t \sigma(X_s)d{B}_s + \int_0^t b(X_s)ds,\ \ t\in[0,T]
\end{equation}
for a finite $T>0$ and $x_0\in\mathbb R,$ where $\sigma$ and $b$ are  Borel measurable functions.

\subsection{Assumptions}

$(H1)$:
There exists some $y_0\in\mathbb R$ and $\varepsilon>0$ 
such that
$\sigma$ and $b$ are bounded on the open ball 
$
B_{6\varepsilon}(y_0):=\{y\in\mathbb{R};|y-y_0|<6\varepsilon\}
$. 
Moreover, 
$\displaystyle
\inf_{x\in B_{6\varepsilon}(y_0) }|\sigma(x)|>\sigma_0>0
$
for some constant $\sigma_0$.
\\
$(H2)$
$\sigma\in C^\infty _b(B_{6\varepsilon}(y_0))$.
\\ 
$(H3)$:
$\displaystyle\sigma^{-1}b:=\frac{b}{\sigma}$ is $\alpha$-H\"{o}lder continuous on $
B_{6\varepsilon}(y_0),
$ 
where $\alpha\in(0,1)$.
\begin{remark}
The assumption $(H3)$ implies that the function $b$ is $\alpha$-H\"older continuous if 
$\sigma$ belongs to $C_b^1$. 
\end{remark}
If our assumptions $(H1),$ $(H2)$ and $(H3)$ are satified on $\mathbb R$, coefficients $\sigma$ and $b$ also satisfy the assumptions in Fournier and Primtems \cite{F1}. 
However their method does not apply if one wants to study the smoothness of the density.  

We assume throughout the article the weak existence of solutions for (\ref{SDE}). Sufficient conditions are stated in e.g. \cite{Cherny}. Our main result is the following theorem.
\begin{theorem}
\label{main}
Assume $(H1), (H2)$ and $(H3)$. 
Then for any initial value $x_0$, any $0<t\le T$ and any $0<\gamma<\alpha,$ 
the distribution of $X_t$ has a $\gamma$-H\"older continuous density on $B_{\varepsilon}(y_0)$. 
\end{theorem}
\begin{remark}
We define 
\begin{align*}
I:=\{y\in\mathbb R;P(T_y<\infty)>0\},
\end{align*}
where
\begin{align*}
T_y:=\inf \{t>0;X_t=y\}.
\end{align*}
Then $I$ forms an interval when $I$ is not a point $($see Section $3.5$ $($page $92)$ in It\^o-McKean\cite{I1}$)$. 
The process $X$ does not go out from $I$, hence the support of the distribution of $X_t$ is contained in the closure of $I$. 
Thus we may concentrate our attention on the interval $I$, although in assumption $(H1)$ we may pick $y_0\in\mathbb R$ which belongs to the complement of $I$ and obtain the existence of a density (which is zero). 
\end{remark}
\section{Estimate of the characteristic function}
We assume without loss of generality that the $\alpha$-H\"older continuity constant of $\sigma^{-1}b$ is equal to one. 
Now we start the study of the characteristic function of $X_t$. 

\subsection{Change of the measure and localization}
Fix $0<t<T.$

We define the coefficients 
$
\bar\sigma(y):=\sigma(\lambda(y))
$ and
$
\bar b(y):=b(\lambda(y))
$
where $\lambda\in C_b^{\infty}$ (a truncation function) is defined by 
\begin{align*}
\lambda(y)=
\left\{ 
\begin{array}{ll}
y;\ \ &\text {if } |y-y_0|\le 4\varepsilon \\
y_0 + 5\varepsilon\frac{y-y_0}{|y-y_0|};\ &\text{if } |y-y_0|\ge 5\varepsilon \\
\end{array}
\right.
\end{align*}
and $\lambda(y)\in\overline{B_{5\varepsilon}(y_0)}$ for all $y\in\mathbb R$. 
As a consequence of $(H1)$ and $(H2)$, 
$\bar\sigma$ is an $C_b^{\infty}$ extension of $\sigma|_{B_{4\varepsilon}(y_0)}$ and 
$\bar\sigma^{-1}\bar b$ is $\alpha$-H\"older continuous on $\mathbb R$.

Let $0<\delta<(t\wedge 1)$. Define
\begin{equation}\label{loc_SDE}
\bar X_s(v,y):= y + \int_v^s\bar\sigma(\bar X_u(v,y))dB_u + \int_v^s\bar b(\bar X_u(v,y))du,
\end{equation}
\begin{equation*}
\nu :=\inf\{s\ge t-\delta;X_s\in\overline{B_{3\varepsilon}(y_0)}\}
\end{equation*}
and
\begin{equation*}
\tau :=\inf\{s\ge \nu;X_s\notin\overline{B_{4\varepsilon}(y_0)}\}.
\end{equation*}
Define the sets
\begin{align*}
A:=\{ \phi_{\varepsilon}(X_t-y_0)>0;\nu = t - \delta,t<\tau \} 
\end{align*}
and
\begin{flalign*}
C:=\{ \phi_{\varepsilon}(X_t-y_0)>0;\sup_{0\le s \le \delta}|\bar X_{\nu+s}(\nu,X_{\nu})-X_{\nu}|\ge\varepsilon \}
\setminus A.
\end{flalign*}
Then we have 
$
\{\phi_{\varepsilon}(X_t-y_0)>0\}=A\cup C.
$
Hence, as $A\cap C=\emptyset$, then
\begin{flalign}
\label{eq:dec}
\!\!\!\! E_Q[e^{i\theta X_t}\phi_{\varepsilon}(X_t-y_0)]
&\! =\! E_Q[e^{i\theta X_t}\phi_{\varepsilon}(X_t-y_0)1_C]+E_Q[e^{i\theta X_t}\phi_{\varepsilon}(X_t-y_0)1_A].
\end{flalign}

The next step in the proof is to remove the coefficient $\bar b$ from $(\ref{loc_SDE})$ in the case of 
$(v,y)=(t-\delta,X_{t-\delta})$ 
by changing the measure. 
Define the stochastic processes for $ t-\delta\le s\le T$ 
\begin{equation*}
W_s:=B_s + \int_{t-\delta}^s(\bar\sigma^{-1}\bar b)(\bar X_u(t-\delta,X_{t-\delta}))du,
\end{equation*}
\begin{equation*}
Z_s:=\exp\left(\int_{t-\delta}^s(\bar\sigma^{-1}\bar b)(\bar X_u(t-\delta,X_{t-\delta}))dB_u 
+ 
\frac{1}{2}\int_{t-\delta}^s|(\bar\sigma^{-1}\bar b)(\bar X_u(t-\delta,X_{t-\delta}))|^2du\right)\
\end{equation*}
and introduce the probability measure $P$ as
\begin{equation}\label{measure}
\frac{dP}{dQ}\Big|_{\mathcal F_s}=Z_s^{-1}\ (t-\delta\le s\le T).
\end{equation}
Then $\bar X(t-\delta,X_{t-\delta})$ satisfies the following SDE;
\begin{equation*}
\bar X_s(t-\delta,X_{t-\delta})=X_{t-\delta} + \int_{t-\delta}^s\bar\sigma(\bar X_u(t-\delta,X_{t-\delta}))dW_u.
\end{equation*} 
\begin{remark}
Due to $(H1)$ and the boundedness of $\bar\sigma^{-1}\bar b$, $Z^{-1}$ satisfies the Novikov condition.
Hence, under the measure P, $W$ is a one dimensional Wiener process. In order to apply Malliavin Calculus in the setting given in Section \ref{intM} we may change probability spaces without any further mention.
\end{remark}
Let us remark some general properties of stochastic processes of exponential type.
\begin{lem}\label{Lem2}$Z$ satisfies the following SDE:
\begin{equation}\label{Z1}
Z_t=1+\int_{t-\delta}^tZ_s(\bar\sigma^{-1}\bar b)(\bar X_s(t-\delta,X_{t-\delta}))dW_s.
\end{equation}
In general, for predictable bounded processes $\psi$ $($the lowest upper bound is denoted by $\|\psi\|_{\infty}),$ we have that processes of the type 
\begin{equation*}\label{Z2}
Z_t=1+\int_{t-\delta}^tZ_s\psi(s)dW_s=\exp\left(\int_{t-\delta}^t\psi(s)dW_s - \frac{1}{2}\int_{t-\delta}^t|\psi(s)|^2ds\right)
\end{equation*}
satisfy that
$$
E[Z_t^p]
\le
\exp\left(\frac{p(p-1)}{2}\delta\|\psi\|_{\infty}^2\right).
$$
\end{lem}

\begin{proof}
For the first property, it is enough to note that $dB_s=dW_s + \bar\sigma^{-1}\bar b(\bar X_s(t-\delta,X_{t-\delta}))ds$. 

Since $W$ is a Wiener process under $P$, $Z$ is a $\mathcal{F}$-martingale under $P$ and hence for any $p>1,$
\begin{align*}
E[Z_t^p]
&\le
E\left [\exp\left(\int_{t-\delta}^t p\psi(s)dW_s - \frac{1}{2}\int_{t-\delta}^t|p\psi(s)|^2ds
+\frac{p(p-1)}{2}\int_{t-\delta}^t|\psi(s)|^2ds\right)\right ]\\
&\le
\exp\left(\frac{p(p-1)}{2}\delta\|\psi\|_{\infty}^2\right).
\end{align*}
\end{proof}
\subsection{Proof of the main theorem}

\begin{proof}
Now we turn to the proof of Theorem 1. 
By Lemmas \ref{lem:EST1} and \ref{lem:EA3} in the Appendix applied to (\ref{eq:dec}), 
we obtain that for some positive constants $K_n,\ M_n,\ C_{\varepsilon,n_2},\ C_{\alpha}$ and $\tilde C_{\varepsilon,n_2}$ 
the following inequality is satisfied
\begin{flalign}\label{EST7}
&|E_Q[e^{i\theta X_t}\phi_{\varepsilon}(X_t-y_0)]|\\
\nonumber&\le
2\varepsilon^{-2n}K_n\left(M_n\|\bar\sigma\|^{2n}_{\infty}\delta^{n}+\delta^{2n}\|\bar b\|_{\infty}^{2n}\right)\\
\nonumber
&+C_{\varepsilon,n_2}|\theta\delta^{\frac{1}{2}}|^{-n_2}
+C_{\alpha}\delta^{\frac{1+\alpha}{2}}
+\|\bar\sigma^{-1}\bar b\|_{\infty}\tilde C_{\varepsilon,n_2}|\theta\delta^{\frac{1}{2}}|^{-n_2},
\end{flalign}
Since $\delta \in(0,t\wedge 1)$ is an arbitrary number, 
we can take 
\begin{equation*}
\delta:=|\theta|^{-\beta},
\end{equation*}
for $|\theta|>(t\wedge 1)^{-\frac 1\beta}$ and any $\beta>0.$
If we denote by $\bar {C}_{\varepsilon ,n,n_2}$ the maximum of all the coefficients of $\theta$ appearing in 
 (\ref{EST7}), we rewrite that inequality as
\begin{flalign*}
&|E_Q[e^{i\theta X_t}\phi_{\varepsilon}(X_t-y_0)]|
\le
\bar {C}_{\varepsilon ,n,n_2}
(
|\theta|^{-n\beta}
+
|\theta|^{-2n\beta}
+
|\theta|^{-\frac{(2-\beta)n_2}{2}}
+
|\theta|^{-\frac{(1+\alpha)\beta}{2}}
)
\end{flalign*}
for $(t\wedge 1)^{-\frac 1\beta}<|\theta|$ and $\frac{2}{1+\alpha}<\beta<2$.
Since $n$ and $n_2$ are arbitrary, 
if we choose $\gamma\in(0,\alpha)$, 
$\beta$ as
\begin{equation*}
\frac{2(1+\gamma)}{1+\alpha}<\beta <2
\end{equation*}
and sufficiently large $n$ and $n_2$, then
by (\ref{Goal}), $X_t$ has a $\gamma$-H\"older continuous density on $B_{\varepsilon}(y_0).$
\end{proof}
\begin{remark}
1. Note that as $\beta$ is chosen closer to $2$, $n_2$ has to be chosen bigger. Therefore in comparison with the classical proofs of the regularity of the density, we need higher regularity of $\sigma$ in order to obtain $\gamma$-H\"older properties of the density for $\gamma$ closer to $\alpha$. 

2. Note that the term that decided the rate of decrease for the localized characteristic function was $\delta^{\frac{1+\alpha}2}$ which is the approximation term for the Girsanov change of measure and which strongly uses the H\"older continuity of $b$ $($see the proof of Lemma $\ref{lem:EA3}$. Therefore even if the other terms may have a faster rate of decrease this will not improve the final result. 
\end{remark}
\section{Conclusions}
We have proved that the regularity of the diffusion coefficient can help transfer the irregularity of the drift to the density function in contrast to the role played by the drift in \cite{F1} and \cite{S1}. In both of these results the drift seems does not seem to play any important role. In this article, we intended to point out that this is not the case and that the regularity of the drift may play an important role in determining the regularity of the density. This is the point where the integration by parts formula of Malliavin Calculus plays an important role in comparison with the previously mentioned results. 

In fact, in a related research, we intend to show, using a more complicated technique ( this involves a more complex version of the technique introduced in \cite{K2}) in the case that the diffusion coefficient is constant, that there are situations where the drift is the determining factor in the regularity of the density of $X_t$.
\section{Appendix}
\subsection{Estimate of (\ref{eq:dec}) on the event C}
\ 

\begin{lem}
\label{lem:EST1}
Under $(H1)$ and $(H2)$, we have the following estimate:
\begin{equation}
\label{EST1}
|E_Q[e^{i\theta X_t}\phi_{\varepsilon}(X_t-y_0)1_C]|
\le
\varepsilon^{-2n}K_n\left(M_n\|\bar\sigma\|^{2n}_{\infty}\delta^{n}+\delta^{2n}\|\bar b\|_{\infty}^{2n}\right),
\end{equation}
where $K_n$ and $M_n$ are constants depend only on $n$.
\end{lem}
\begin{proof} 
Using Markov's inequality, we have
\begin{flalign*}
Q(C)&\le Q\left(\sup_{0\le s \le \delta}|\bar X_{\nu+s}(\nu,X_{\nu})- X_{\nu}|\ge\varepsilon\right)\\
&\le \varepsilon^{-2n}E_Q\left[\sup_{0\le s \le \delta}|\bar X_{\nu+s}(\nu,X_{\nu})- X_{\nu}|^{2n}\right]\\
&\le \varepsilon^{-2n}K_n\left(E_Q\left[\sup_{0\le s \le \delta}|\int_{\nu}^{s+\nu}\bar\sigma(\bar X_u(\nu,X_{\nu}))dB_u|^{2n}\right]
+E_Q\left[\sup_{0\le s \le \delta}|\int_{\nu}^{s+\nu}\bar b(\bar X_u(\nu,X_{\nu}))du|^{2n}\right]\right),
\end{flalign*}
where $K_n$ is a constant which depends only on $n.$ 

Since $\bar\sigma$ and $\bar b$ are bounded, 
by Doob's inequality and Burkholder-Davis-Gundy inequality, we have
\begin{align*}
&\varepsilon^{-2n}K_n\left(E_Q\left[\sup_{0\le s \le \delta}|\int_{\nu}^{s+\nu}\bar\sigma(\bar X_u(\nu,X_{\nu}))dB_u|^{2n}\right]
+E_Q\left[\sup_{0\le s \le \delta}|\int_{\nu}^{s+\nu}\bar b(\bar X_u(\nu,X_{\nu}))du|^{2n}\right]\right)\\
&\le
\varepsilon^{-2n}K_n\left(M_nE_Q\left[\left\{\int_\nu^{\nu+\delta}(\bar\sigma(\bar X_u(\nu,X_{\nu}))^2du\right\}^{n}\right]
+(\delta\|\bar b\|_{\infty})^{2n}\right)\\
&\le
\varepsilon^{-2n}K_n\left(M_n\|\bar\sigma\|^{2n}_{\infty}\delta^{n}
+(\delta\|\bar b\|_{\infty})^{2n}\right)
\end{align*}
for any $n\in\mathbb{N}$, where
$M_n$ is a constant depends only on $n.$ 
Therefore $(\ref{EST1})$ follows.
\end{proof}
\subsection{Estimate of (\ref{eq:dec}) on the event A}
\ 

Now we turn to estimate the second term of (\ref{eq:dec}).
\begin{lem}
\label{lem:EA3}
Under $(H1),(H2)$ and $(H3)$, we have the following estimate:
\begin{align*}
|E_Q[e^{i\theta X_t}\phi_{\varepsilon}(X_t-y_0)1_A]|
&\le
\varepsilon^{-2n}K_n\left(M_n\|\bar\sigma\|^{2n}_{\infty}\delta^{n}+\delta^{2n}\|\bar b\|_{\infty}^{2n}\right)\\
&+
C_{\varepsilon,n_2}|\theta\delta^{\frac{1}{2}}|^{-n_2}
+
C_{\alpha}\delta^{\frac{1+\alpha}{2}}+
\|\bar\sigma^{-1}\bar b\|_{\infty}\tilde C_{\varepsilon,n_2}|\theta\delta^{\frac{1}{2}}|^{-n_2}.
\end{align*}
\end{lem}
\begin{proof}
By the definition of $\bar X,$ on the event $A$,
\begin{align*}
X_t=\bar X_t(t-\delta,X_{t-\delta}).
\end{align*}
Hence, we obtain that 
\begin{align*}
|E_Q[e^{i\theta X_t}\phi_{\varepsilon}(X_t-y_0)1_A]|
&=|E_Q[e^{i\theta \bar X_t(t-\delta,X_{t-\delta})}\phi_{\varepsilon}(\bar X_t(t-\delta,X_{t-\delta}) - y_0)1_A]|.
\end{align*}

Since 
\begin{flalign*}
1_{\{\nu=t-\delta;t<\tau\}}
&=1_{\{\nu=t-\delta;t<\tau\}}1_{\{X_{t-\delta}\in\overline{B_{3\varepsilon}(y_0)}\}}\\
&=(1 - 1_{\{\nu=t-\delta;\tau\le t\}})1_{\{X_{t-\delta}\in\overline{B_{3\varepsilon}(y_0)}\}}\\
&=1_{\{X_{t-\delta}\in\overline{B_{3\varepsilon}(y_0)}\}} - 1_{\{\nu=t-\delta;\tau\le t\}},
\end{flalign*} 
we have
\begin{flalign}\label{Ineq1}
&|E_Q[e^{i\theta \bar X_t(t-\delta,X_{t-\delta})}\phi_{\varepsilon}(\bar X_t(t-\delta,X_{t-\delta}) - y_0)1_A]|\\
\nonumber&\le
|E_Q[e^{i\theta \bar X_t(t-\delta,X_{t-\delta})}\phi_{\varepsilon}(\bar X_t(t-\delta,X_{t-\delta}) - y_0)1_{\{X_{t-\delta}\in\overline{B_{3\varepsilon}(y_0)}\}}]|\\
\nonumber&+
|E_Q[e^{i\theta \bar X_t(t-\delta,X_{t-\delta})}\phi_{\varepsilon}(\bar X_t(t-\delta,X_{t-\delta}) - y_0)1_{\{\nu=t-\delta;\tau\le t\}}]|.
\end{flalign}

By the definitions of $\nu$ and $\tau$, we have
\begin{align*}
\{\nu=t-\delta;\tau\le t\}
\subseteq 
\{\sup_{0\le s \le \delta}|\bar X_{t-\delta +s}(t-\delta,X_{t-\delta})- X_{t-\delta}|\ge\varepsilon\}.
\end{align*}

So, as in Lemma \ref{lem:EST1} we obtain that 
\begin{equation*}
Q(\nu=t-\delta;\tau\le t)\le
\varepsilon^{-2n}K_n\left(M_n\|\bar\sigma\|^{2n}_{\infty}\delta^{n}
+\delta^{2n}\|\bar b\|_{\infty}^{2n}\right).
\end{equation*}

Therefore, we have the following upper bound for the second term in $(\ref{Ineq1})$
\begin{align}\label{EST2}
&|E_Q[e^{i\theta \bar X_t(t-\delta,X_{t-\delta})}\phi_{\varepsilon}(\bar X_t(t-\delta,X_{t-\delta}) - y_0)1_{\{\nu=t-\delta;\tau\le t\}}]|\\
\nonumber&\le
\varepsilon^{-2n}K_n\left(M_n\|\bar\sigma\|^{2n}_{\infty}\delta^{n}
+\delta^{2n}\|\bar b\|_{\infty}^{2n}\right).
\end{align}

For the first term in $(\ref{Ineq1})$, we change the probability measure from $Q$ to $P$ defined by ($\ref{measure}$). That is,
\begin{align*}
&E_Q[e^{i\theta \bar X_t(t-\delta,X_{t-\delta})}\phi_{\varepsilon}(\bar X_t(t-\delta,X_{t-\delta}) - y_0)1_{\{X_{t-\delta}\in\overline{B_{3\varepsilon}(y_0)}\}}]\\
&=
E_P[e^{i\theta \bar X_t(t-\delta,X_{t-\delta})}\phi_{\varepsilon}(\bar X_t(t-\delta,X_{t-\delta}) - y_0)Z_t1_{\{X_{t-\delta}\in\overline{B_{3\varepsilon}(y_0)}\}}].
\end{align*} 
Then we have
\begin{align}\label{EA1}
&|E_P[e^{i\theta \bar X_t(t-\delta,X_{t-\delta})}\phi_{\varepsilon}(\bar X_t(t-\delta,X_{t-\delta}) - y_0)Z_t1_{\{X_{t-\delta}\in\overline{B_{3\varepsilon}(y_0)}\}}]|\\
\nonumber&\le
|E_P[e^{i\theta \bar X_t(t-\delta,X_{t-\delta})}\phi_{\varepsilon}(\bar X_t(t-\delta,X_{t-\delta}) - y_0)
(Z_t-1)1_{\{X_{t-\delta}\in\overline{B_{3\varepsilon}(y_0)}\}}]|\\
\nonumber&+
|E_P[e^{i\theta \bar X_t(t-\delta,X_{t-\delta})}\phi_{\varepsilon}(\bar X_t(t-\delta,X_{t-\delta}) - y_0)
1_{\{X_{t-\delta}\in\overline{B_{3\varepsilon}(y_0)}\}}]|.
\end{align}

Since $1_{\{X_{t-\delta}\in\overline{B_{3\varepsilon}(y_0)}\}}$ is $\mathcal F_{t-\delta}$-measurable, 
using conditional expectation and the Markov property for $\bar X$, we have 
\begin{flalign*}
&|E_P[e^{i\theta \bar X_t(t-\delta,X_{t-\delta})}\phi_{\varepsilon}(\bar X_t(t-\delta,X_{t-\delta}) - y_0)1_{\{X_{t-\delta}\in\overline{B_{3\varepsilon}(y_0)}\}}]|\\
&=|E_P[E_P[e^{i\theta \bar X_t(t-\delta,X_{t-\delta})}\phi_{\varepsilon}(\bar X_t(t-\delta,X_{t-\delta}) - y_0)|\mathcal F_{t-\delta}]1_{\{X_{t-\delta}\in\overline{B_{3\varepsilon}(y_0)}\}}]|\\
&\le
\sup_{y\in\overline{B_{3\varepsilon}(y_0)}}|E_P[e^{i\theta \bar X_t(t-\delta,y)}\phi_{\varepsilon}(\bar X_t(t-\delta,y) - y_0)]|.
\end{flalign*}

As in Proposition $\ref{Int_By_Parts}$, the integration by parts formula of Malliavin calculus in the interval $[t-\delta ,t]$, implies that 
for any $n_2\in\mathbb{N}$ and $y\in\overline{B_{3\varepsilon}(y_0)},$ there exists a random variable $H_{n_2}(\bar X_t(t-\delta,y),\phi_{\varepsilon}(\bar X_t(t-\delta,y) - y_0))\in\mathbb{D^{\infty}}$ such that 
\begin{flalign*}
&E_P\left [\frac{d^{n_2}}{d x^{n_2}}\left(e^{i\theta x}\right)\bigg|_{x=\bar X_t(t-\delta,y)}\phi_{\varepsilon}(\bar X_t(t-\delta,y) - y_0)\right ]\\
&=E_P\left [e^{i\theta X_t(t-\delta,y)}H_{n_2}(\bar X_t(t-\delta,y),\phi_{\varepsilon}(\bar X_t(t-\delta,y) - y_0))\right ].
\end{flalign*}

Furthermore, by Theorem $2.3.$ and Corollary $1$ of \cite{S2} (which are consequences of the application of Proposition \ref{Int_By_Parts} to our situation), there exists a constant $C_{\varepsilon,n_2}$ which depends on $\varepsilon,$ $n_2$ and 
derivatives of $\bar\sigma$ up to the order $n_2$ such that for any $y\in\overline{B_{3\varepsilon}(y_0)},$
\begin{equation}\label{IBP1}
\|H_{n_2}(\bar X_t(t-\delta,y),\phi_{\varepsilon}(\bar X_t(t-\delta,y) - y_0))\|_{L^2(P)}
\le C_{\varepsilon,n_2}\delta^{-\frac{n_2}{2}}.
\end{equation}
In fact,
Theorem $2.3.$ of $\cite{S2}$ tells us that there exists some constant $C^{\star}_{\varepsilon ,n}$ 
such that
\begin{equation*}
\|H_{n_2}(\bar X_t(t-\delta,y),\phi_{\varepsilon}(\bar X_t(t-\delta,y) - y_0))\|_{L^2(P)}
\le
C^{\star}_{\varepsilon ,n_2}\|\phi_{\varepsilon}(\bar X_t(t-\delta,y) - y_0)\|_{n_2,2^{n_2+1}}\delta^{-\frac{n_2}{2}}.
\end{equation*}

On the other hand, thanks to $(H1)$, 
Corollary $1$ of \cite{S2} implies that there exists some constant $C^{\dagger}_{\varepsilon ,n_2}$ 
such that
\begin{equation*}
\|\phi_{\varepsilon}(\bar X_t(t-\delta,y) - y_0)\|_{n_2,2^{n_2+1}}
\le
C^{\dagger}_{\varepsilon ,n_2}.
\end{equation*}

The above constant $C_{\varepsilon ,n_2}$ is the product of these constants 
$
C^{\star}_{\varepsilon ,n_2}
$ and 
$
C^{\dagger}_{\varepsilon ,n_2}
$.

By (\ref{IBP1}) and recalling that $Z$ is a non-negative martingale with mean one, for any $n_2\in\mathbb{N},$ we obtain the following inequality.
\begin{flalign}\label{EST3}
&|E_P[e^{i\theta \bar X_t(t-\delta,X_{t-\delta})}\phi_{\varepsilon}(\bar X_t(t-\delta,X_{t-\delta}) - y_0)1_{\{X_{t-\delta}\in\overline{B_{3\varepsilon}(y_0)}\}}]|\\
\nonumber&\le
C_{\varepsilon,n_2}|\theta\delta^{\frac{1}{2}}|^{-n_2}.
\end{flalign}

However, since $Z_t-1$ is not $\mathcal{F}_{t-\delta}$-measurable and 
we do not assume the smoothness of the coefficient $b,$ 
we can not apply the integration by parts formula for the first term in $(\ref{EA1})$. Instead, we rewrite
 \begin{align*}
Z_t-1
=&\int_{t-\delta}^t(\bar\sigma^{-1}\bar b)(\bar X_u(t-\delta,X_{t-\delta}))Z_u-(\bar\sigma^{-1}\bar b)(X_{t-\delta})dW_u\\
&+(\bar\sigma^{-1}\bar b)(X_{t-\delta})(W_t-W_{t-\delta}).
\end{align*}
Thus we obtain
\begin{align}\label{Ineq2}
&\left|E_P\left[e^{i\theta \bar X_t(t-\delta,X_{t-\delta})}\phi_{\varepsilon}(\bar X_t(t-\delta,X_{t-\delta}) - y_0)(Z_t-1)1_{\{X_{t-\delta}\in\overline{B_{3\varepsilon}(y_0)}\}}\right]\right|\\
\nonumber&\le
E_P\left[\int_{t-\delta}^t|(\bar\sigma^{-1}\bar b)(\bar X_u(t-\delta,X_{t-\delta}))Z_u-(\bar\sigma^{-1}\bar b)(X_{t-\delta})|^2du\right]^{\frac{1}{2}}\\
\nonumber&+
|E_P[e^{i\theta \bar X_t(t-\delta,X_{t-\delta})}\phi_{\varepsilon}(\bar X_t(t-\delta,X_{t-\delta}) - y_0)(\bar\sigma^{-1}\bar b)(X_{t-\delta})(W_t-W_{t-\delta})1_{\{X_{t-\delta}\in\overline{B_{3\varepsilon}(y_0)}\}}]|.
\end{align}

For the first term, by the H\"older continuity of $\bar\sigma^{-1}\bar b$, $(\ref{Z1})$ and H\"older's inequality, we have 
\begin{align}\label{Ineq3}
&E_P\left[\int_{t-\delta}^t|(\bar\sigma^{-1}\bar b)(\bar X_u(t-\delta,X_{t-\delta}))Z_u-(\bar\sigma^{-1}\bar b)(X_{t-\delta})|^2du\right]^{\frac{1}{2}}\\
\nonumber&\le
\sqrt 2\bigg[\int_{t-\delta}^tE_P\left[|(\bar\sigma^{-1}\bar b)(\bar X_u(t-\delta,X_{t-\delta}))-(\bar\sigma^{-1}\bar b)(X_{t-\delta})|^2Z_u^2\right]du\\
\nonumber&+
\int_{t-\delta}^tE_P[|(\bar\sigma^{-1}\bar b)(X_{t-\delta})|^2(Z_u-1)^2]du\bigg]^{\frac{1}{2}}\\
\nonumber&\le
\sqrt2\bigg[
\int_{t-\delta}^tE_P[|\bar X_u(t-\delta,X_{t-\delta})-X_{t-\delta}|^{2\alpha}Z_u^2]du\\
\nonumber&+
\|\bar\sigma^{-1}\bar b\|_{\infty}^2\int_{t-\delta}^t\int_{t-\delta}^uE_P[|(\bar\sigma^{-1}\bar b)(\bar X_v(t-\delta,X_{t-\delta}))Z_v|^2]dvdu
\bigg]^{\frac{1}{2}}\\
\nonumber&\le
\sqrt2\left[
\int_{t-\delta}^tE_P[|\bar X_u(t-\delta,X_{t-\delta})-X_{t-\delta}|^2]^{\alpha}E_P[Z_u^{\frac{2}{1-\alpha}}]^{1-\alpha}du
+
\frac{\|\bar\sigma^{-1}\bar b\|_{\infty}^4}{2}\|Z_t\|_{L^2(P)}^2\delta^2
\right]^{\frac{1}{2}}\\
\nonumber&\le
\sqrt 2\left[\frac 2{1+\alpha}
\|Z_t^2\|_{L^{\frac{1}{1-\alpha}}(P)} \|\bar\sigma\|_{\infty}^{2\alpha}\delta^{1+\alpha}
+\frac{\|\bar\sigma^{-1}\bar b\|_{\infty}^4}{2}\|Z_t\|_{L^2(P)}^2\delta^2
\right]^{\frac{1}{2}}\\
\nonumber&\le
C_{\alpha}\delta^{\frac{1+\alpha}{2}},
\end{align}
where 
\begin{equation*}
C_{\alpha}:=
\left(\frac {2}{\sqrt{1+\alpha}}\|Z_t\|_{L^{\frac{2}{(1-\alpha)}}(P)} \|\bar\sigma\|_{\infty}^{\alpha}\right)
\vee
\left(\|\bar\sigma^{-1}\bar b\|_{\infty}^2\|Z_t\|_{L^2(P)}\right).
\end{equation*}

For the second term of (\ref{Ineq2}), we proceed as  in (\ref{EST3}). 
Since $(\bar\sigma^{-1}\bar b)(X_{t-\delta})$ is bounded  and $\mathcal{F}_{t-\delta}$-measurable,  
we have 
\begin{align*}
&|E_P[e^{i\theta \bar X_t(t-\delta,X_{t-\delta})}\phi_{\varepsilon}(\bar X_t(t-\delta,X_{t-\delta}) - y_0)(\bar\sigma^{-1}\bar b)(X_{t-\delta})(W_t-W_{t-\delta})1_{\{X_{t-\delta}\in\overline{B_{3\varepsilon}(y_0)}\}}]|\\
&\le
\|\bar\sigma^{-1}\bar b\|_{\infty}
\sup_{y\in\overline{B_{3\varepsilon}(y_0)}}|E_P[e^{i\theta \bar X_t(t-\delta,y)}\phi_{\varepsilon}(\bar X_t(t-\delta,y) - y_0)(W_t-W_{t-\delta})]|.
\end{align*}

Now we can apply the integration by parts formula which implies that 
for any $n_2\in\mathbb N$ and $y\in\overline{B_{3\varepsilon}(y_0)}$
there exists a random variable 

$H_{n_2}(\bar X_t(t-\delta,y),\phi_{\varepsilon}(\bar X_t(t-\delta,y) - y_0)(W_t-W_{t-\delta}))\in\mathbb{D}^{\infty}$ 
such that 
\begin{align*}
&E_P\left [\frac{d^{n_2}}{dx^{n_2}}(e^{i\theta x})\bigg|_{x=\bar X_t(t-\delta,y)}\phi_{\varepsilon}(\bar X_t(t-\delta,y) - y_0)(W_t-W_{t-\delta})\right]\\
&=
E_P\left [e^{i\theta \bar X_t(t-\delta,y)}H_{n_2}(\bar X_t(t-\delta,y),\phi_{\varepsilon}(\bar X_t(t-\delta,y ) - y_0)(W_t-W_{t-\delta}))\right ] 
\end{align*}
and by the H\"older inequality for the stochastic Sobolev norms (see Proposition $1.5.6$ of \cite{N1}), its $L^2(P)$-norm is bounded by 
$
C_{\varepsilon,n_2}\delta^{-\frac{n_2}{2}}
c_{n_2}\|(W_t-W_{t-\delta})\|_{n_2,2^{n_2+1}},
$ 
where $c_{n_2}$ is a constant depends only on $n_2.$  

However, the $k$-th order $H$-derivatives of $W_t-W_{t-\delta}$ vanish when $k\ge 2.$ 
Therefore, there exists a positive constant $C$ (independent of $n_2$) such that 
\begin{equation}\label{IBP2}
\|(W_t-W_{t-\delta})\|_{n_2,2^{n_2+1}}
=
\|(W_t-W_{t-\delta})\|_{1,2^{n_2+1}}
\le
C
\end{equation}
and hence, we have
\begin{align}\label{EST6}
&|E_P[e^{i\theta \bar X_t(t-\delta,X_{t-\delta})}\phi_{\varepsilon}(\bar X_t(t-\delta,X_{t-\delta}) - y_0)(Z_t-1)1_{\{X_{t-\delta}\in\overline{B_{3\varepsilon}(y_0)}\}}]|\\
\nonumber&\le
C_{\alpha}\delta^{\frac{1+\alpha}{2}}+
\|\bar\sigma^{-1}\bar b\|_{\infty}\tilde C_{\varepsilon,n_2}|\theta\delta^{\frac{1}{2}}|^{-n_2},
\end{align}
where 
$
\tilde C_{\varepsilon,n_2}:=2C_{\varepsilon,n_2}c_{n_2}.
$

Substituting $(\ref{EST6})$ and $(\ref{EST3})$ into $(\ref{EA1})$, we have 
\begin{align}\label{EA2}
&|E_P[e^{i\theta \bar X_t(t-\delta,X_{t-\delta})}\phi_{\varepsilon}(\bar X_t(t-\delta,X_{t-\delta}) - y_0)Z_t1_{\{X_{t-\delta}\in\overline{B_{3\varepsilon}(y_0)}\}}]|\\
\nonumber&\le
C_{\varepsilon,n_2}|\theta\delta^{\frac{1}{2}}|^{-n_2}
+
C_{\alpha}\delta^{\frac{1+\alpha}{2}}+
\|\bar\sigma^{-1}\bar b\|_{\infty}\tilde C_{\varepsilon,n_2}|\theta\delta^{\frac{1}{2}}|^{-n_2}.
\end{align}

As a result, we have
\begin{align*}
&|E_Q[e^{i\theta \bar X_t(t-\delta,X_{t-\delta})}\phi_{\varepsilon}(\bar X_t(t-\delta,X_{t-\delta}) - y_0)1_A]|\\
\nonumber&\le
\varepsilon^{-2n}K_n\left(M_n\|\bar\sigma\|^{2n}_{\infty}\delta^{n}+\delta^{2n}\|\bar b\|_{\infty}^{2n}\right)\\
&\nonumber+
C_{\varepsilon,n_2}|\theta\delta^{\frac{1}{2}}|^{-n_2}
+
C_{\alpha}\delta^{\frac{1+\alpha}{2}}+
\|\bar\sigma^{-1}\bar b\|_{\infty}\tilde C_{\varepsilon,n_2}|\theta\delta^{\frac{1}{2}}|^{-n_2}
\end{align*}
by substituting $(\ref{EA2})$ and $(\ref{EST2})$ into $(\ref{Ineq1})$.
\end{proof}

\begin{remark}
The above estimate $(\ref{IBP2})$ for the Sobolev norm of the Wiener process is clearly non-optimal. However, as the term appearing in $(\ref{Ineq3})$ decreases slowly, improving the estimate in $(\ref{IBP2})$ will not change the final result. The same comment applies to other terms such as $(\ref{EST2})$.
\end{remark}


\begin{thebibliography}{99}
\bibitem{Cherny}
Singular Stochastic differential equations 
by Alexander S. Cherny and Engelbert.



\bibitem{Friedman}
Partial differential equations of parabolic type
by Friedman



\bibitem{F1}
N.\,Fournier and J.\,Printems,
Absolute continuity for some one dimensional processes,
Bernoulli, 16(2), 2010, 343-360.
\bibitem{I1}
K.\,It\^o and H.\,P.\,McKean,\,Jr.,
Diffusion Processes and their Sample Paths,
Second edition, Springer Verlag, (1974)

\bibitem{K2}
A.\,Kohatsu-Higa and A.\,Tanaka. A Malliavin Calculus method to study densities of additive
functionals of SDE's with irregular drifts. To appear in Annales de l'Institut Henri Poincar\'e, 2011. 
\bibitem{Krylov}
Lectures on Elliptic and Parabolic Equations in Sobolev Spaces
by Krylov


\bibitem{K1}
S.\,Kusuoka,
Existence of densities of solutions of stochastic differential equations by Malliavin calculus,
J. Functional Analysis 258 (2010), 758-784.

\bibitem{Lady}
Linear and quasi-Lineqr equations of parabolic type 
 by Ladyzenskaja Solonnikov and Ural'ceva

\bibitem{S1}
S.\,D.\,Marco,
On Probability Distributions of Diffusions and Financial Models with non-globally smooth coefficients. PhD. Thesis, http://cermics.enpc.fr/~de-marcs/home.html 
\bibitem{S2}
S.\,D.\,Marco,
Smoothness and asymptotic estimates of densities for SDEs with locally smooth coefficients and applications to square root-type diffusions,
2011, The Annals of Applied Probability 21, Number 4 (2011), 1282-1321
\bibitem{N1}
D.\,Nualart,
The Malliavin calculus and Related Topics,
Second edition, Springer Verlag, (2006)
\bibitem{W1}
D.\,Williams,
Probability with martingales, 
Cambridge University, (1991)
\end{thebibliography}
\end{document}